\documentclass[12pt]{amsart}
\usepackage{color}
\usepackage{xypic}
\usepackage[colorlinks=true,urlcolor=blue,linkcolor=blue,citecolor=blue]{hyperref}
\usepackage{ amssymb }

\numberwithin{equation}{section}
\theoremstyle{plain}
\newtheorem{theorem}{Theorem}[section]

\newtheorem{lemma}[theorem]{Lemma}
\newtheorem{cor}[theorem]{Corollary}

\theoremstyle{definition}

\newtheorem{definition}[theorem]{Definition}
\newtheorem{exa}[theorem]{Example}

\newtheorem{que}[theorem]{Question}

\newcommand{\st}{\ :\ }

\newcommand{\A}{\mathbf{A}}
\newcommand{\B}{\mathbf{B}}
\newcommand{\C}{\mathcal{C}}

\newcommand{\M}{\mathbf{M}}
\newcommand{\N}{\mathbb{N}}
\newcommand{\R}{\mathbf{R}}
\renewcommand{\S}{\mathbf{S}}
\newcommand{\Z}{\mathbb{Z}}

\newcommand{\Clo}{\mathrm{Clo}}

\newcommand{\GL}{\mathrm{GL}}
\newcommand{\Pol}{\mathrm{Pol}}
\newcommand{\rk}{\mathrm{rk}}
\newcommand{\Hom}{\mathrm{Hom}}

\title{Clonoids between modules}
\date{\today}
\author[P. Mayr]{Peter Mayr}
\address{Department of Mathematics \\
University of Colorado \\ Boulder CO 80309 \\USA}
\email{peter.mayr@colorado.edu\\ patrick.wynne@colorado.edu}

\author[P. Wynne]{Patrick Wynne}

\thanks{The research of the first author was partially supported by the NSF under grant no.~DMS-1500254  and by the Austrian Science Fund (FWF): P33878.}



\begin{document}

\begin{abstract}
 Clonoids are sets of finitary functions from an algebra $\A$ to an algebra $\B$ that are closed under composition with
 term functions of $\A$ on the domain side and with term functions of $\B$ on the codomain side.
 For $\A,\B$ (polynomially equivalent to) finite modules we show:
 If $\A,\B$ have coprime order and the congruence lattice of $\A$ is distributive, then there are only finitely many clonoids
 from $\A$ to $\B$.
 This is proved by establishing for every natural number $k$ a particular linear equation that 
 all $k$-ary functions from $\A$ to $\B$ satisfy.  
 Else if $\A,\B$ do not have coprime order, then there exist infinite
 ascending chains of clonoids from $\A$ to $\B$ ordered by inclusion.
 Consequently any extension of $\A$ by $\B$ has countably infinitely many $2$-nilpotent expansions up to term
 equivalence.
\end{abstract} 

\maketitle


\section{Introduction}

 Clones are sets of finitary operations on a set that are closed under composition and contain all projection maps.
 They have been used for studying completeness of Boolean operations (Post~\cite{Po:TVIS}),
 classifying algebraic structures (see the books~\cite{FMMT:ALV3,La:FAFS,Sz:CUA}),
 and capturing the computational complexity of Constraint Satisfaction Problems (see the survey~\cite{BKW:PHU}) for many
 decades.

 More recently and more generally, sets of finitary functions from one set to another that satisfy certain closure properties
 have found applications, e.g., for studying equational theories of finite algebras~\cite{AM:FGEC,Ma:VLN},
 classifying expansions of algebras~\cite{Fio:EAS,Kr:CFS},
 and investigating the computational complexity of Promise Constraint Satisfaction Problems (see the survey~\cite{BBKO}).
 These function classes fit into the general framework of \emph{clonoids}: sets of finitary functions from an algebra
 $\A$ to an algebra $\B$ that are closed under composition with term functions of $\A$ on one side and composition with
 term functions of $\B$ on the other.

Let $\N := \{1, 2, 3, \ldots\}$ denote the set of natural numbers.

\begin{definition}\cite[cf. Definition 4.1]{AM:FGEC} \label{def:clonoid}
 Let $\A$ and $\B$ be algebras and let $C \subseteq \bigcup_{n \in\N} B^{A^n}$ be a set of finitary functions from $A$ to $B$.
 We say that $C$ is a \emph{clonoid from $\A$ to $\B$} if
 \[C \Clo(\A) \subseteq C \, \text{ and } \, \Clo(\B) C \subseteq C.\]
 Here $\Clo(\A)$ and $\Clo(\B)$ are the clones of term functions of $\A$ and $\B$, respectively, and
 juxtaposition represents function class composition as in \cite{CF:FCRC}, i.e.
 \begin{align*}
  C \Clo(\A)  =  \{ f(s_1,\dots,s_k) \st k,m\in\N, f\, k\text{-ary in } C,  s_1,\dots,s_k\ & \\    m\text{-ary term functions on } \A \},  \\
 \Clo(\B)C = \{ t(f_1,\dots,f_n) \st n,k\in\N, t\, n\text{-ary term function on } \B & \\ f_1,\dots,f_n\, k\text{-ary in } C \}.
\end{align*}  
\end{definition}

 For example, for abelian groups $\A$ and $\B$, the set $\bigcup_{k\in\N} \Hom(\A^k,\B)$
 of homomorphisms from finite powers of $\A$ to $\B$ forms a clonoid from $\A$ to $\B$.

 The name `clonoid' was originally introduced by Aichinger and the first author of this paper in~\cite{AM:FGEC}
 for the case that $\A$ is a set with no operations. There clonoids were used to represented equational theories and
 to show that every subvariety of a finitely generated variety with cube term is finitely generated.
 For subsequent applications the slight generalization that allows for $\A$ an algebra as in Definition~\ref{def:clonoid}
 now appears to be more desirable. We mention some other contexts in which clonoids occurred.

 Already in 2002, Pippenger developed a Galois correspondence for clonoids between sets, which he called
 \emph{minor-closed classes},
and pairs of relational structures in~\cite{Pi:GTMF}.
The theory was extended by Couceiro and Foldes in~\cite{CF:FCRC}
 to clonoids from an algebra $\A$ to an algebra $\B$, which they called classes of functions that are \emph{stable} under
 right composition with $\Clo(\A)$ and left composition with $\Clo(\B)$ (see Section~\ref{sec:Galois} below).

 In 2018,
 Brakensiek and Guruswami observed that Pippenger's Galois correspondence can be used to analyze Promise
 Constraint Satisfaction Problems via their finitary symmetries, i.e., sets of polymorphisms between relational
 structures. In this context, clonoids from sets $A$ to $B$ have been called~\emph{minions} (see the survey by Barto,
 Bul\'{\i}n, Krokhin, and Opr\v{s}al~\cite{BBKO}). 

 Kreinecker~\cite{Kr:CFS} and Fioravanti~\cite{Fio:CSPFF,Fio:EAS} studied clonoids between modules $\A$ and $\B$,
 which they called \emph{linearly-closed clonoids}, and used them to classify expansions of $\A\times\B$ up to term
 equivalence.
 For algebras $\A_1,\A_2$ on the same universe, $\A_1$ is an \emph{expansion} of $\A_2$ if
 $\Clo(\A_1)\supseteq\Clo(\A_2)$; we say $\A_1$ and $\A_2$ are \emph{term equivalent} if $\Clo(\A_1) = \Clo(\A_2)$.
 
 All of the function classes mentioned above are instances of clonoids from $\A$ to $\B$ in the sense of
 Definition~\ref{def:clonoid} for the appropriate choice of algebras.

 In this paper we collect basic properties and background on clonoids in Section~\ref{Prelims}. 
 Then we mainly investigate clonoids between algebras that are polynomially equivalent to modules
 in the vein of~\cite{Fio:CSPFF,Kr:CFS}.
 Recall that an algebra $\A$ is polynomially equivalent to a module if all basic operations of $\A$ are affine functions
 on that module (see Section~\ref{sec:Malcev}).
Throughout this paper we will consider rings with $1$ that need not be commutative. 
 All modules considered are unitary modules.

 In particular we consider the following problem:
 
\begin{que} \label{MainQuestion}
 Given finite (algebras that are polynomially equivalent to) modules $\A$ and $\B$, how many clonoids are there from
 $\A$ to $\B$? Are they all finitely generated?
\end{que}

 From work of Aichinger and the first author of this paper in~\cite{AM:FGEC} it follows that for $\B$ a finite algebra with
 few subpowers (in particular, for $\B$ polynomially equivalent to a module), all clonoids from a finite algebra $\A$ to
 $\B$ are finitely related and hence there are at most countably infinitely many (see Theorem~\ref{thm:cube}).

 There has been some recent activity in classifying clonoids of Boolean functions in particular.
 Sparks~\cite{Sp:NC} determined the cardinality of clonoids from a finite set $A$ to any 2-element
 algebra $\B$. In particular she showed that there are countably infinitely many clonoids from $A$ with $|A|>1$
 to $(\Z_2,+)$.
 Couceiro and Lehtonen~\cite{CL:SBF} explicitly described all the countably infinitely many clonoids from
 $(\Z_2,x+y+z)$ to itself. Moreover, Lehtonen~\cite{Le:MC,Le:NUC} explicitly described the finitely many clonoids
 between $2$-element algebras $\A$ and $\B$ where $\B$ has a majority, or more generally, near unanimity term.

 Kreinecker~\cite{Kr:CFS} showed that there are countably infinitely many clonoids from $(\Z_p,+)$ to $(\Z_p,+)$ for
 any prime $p$, which he then used to construct infinitely many non-finitely generated clones on $\Z_p^2$ that contain $+$.
For modules $\A$ and $\B$, we say that $f \colon A^k \rightarrow B$ is additive if 
$f(x+y) = f(x) + f(y)$ for all $x,y \in A^k$.
We generalize Kreinecker's result to obtain the following in Section~\ref{sec:number}.

\begin{theorem} \label{thm:notcoprime}
 Let $\A,\B$ be finite modules whose orders are not coprime.
 Then there exists an infinite ascending chain of clonoids from $\A$ to $\B$ (that contain all additive
 functions from $\A$ to $\B$).
\end{theorem}

Together with Lemma~\ref{lem:ABexpansion} below and the subsequent remark, this yields
 a result of Idziak~\cite{Id:CCMO}
that any finite module $\A$ with a submodule $B$ and $\gcd(|A/B|,|B|) > 1$ has countably
 infinitely many $2$-nilpotent expansions up to term equivalence.

Fioravanti proved for $\A$ and $\B$ finite coprime regular modules over products of finite fields that every clonoid
from $\A$ to $\B$ is generated by its unary functions (and so there are finitely many of
them)~\cite[Theorem 1.2]{Fio:CSPFF}. 
 Note that under his assumptions the lattices of submodules of $\A$ and $\B$ are distributive (in fact, Boolean).
 As a consequence, he then showed that a finite abelian group has finitely many expansions up to term equivalence
 if and only if the order of that group is squarefree in~\cite{Fio:EAS}.
 
 Following \cite{Co:FIR,Tu:SSM}, a module is \emph{distributive} if its lattice of submodules is distributive.
 Our main result generalizes Fioravanti's result and settles Question~\ref{MainQuestion} for $\A$ (polynomially
 equivalent to) a distributive module.
 
\begin{theorem} \label{thm:distributive}
 Let $\A$ be polynomially equivalent to a finite distributive $\R$-module, let $n$ be the nilpotence degree of the
 Jacobson radical of $\R$, and let $\B$ be polynomially equivalent to an $\S$-module such that $|A|$ is
 invertible in $\S$.
\begin{enumerate}
\item    
 Then every clonoid from $\A$ to $\B$ is generated by its $n+1$-ary functions (by its $n$-ary functions if $\A$ is an
 $\R$-module).
\item
 If $\B$ is finite, then there are only finitely many clonoids from $\A$ to $\B$. 
\end{enumerate}
\end{theorem}

 For example, let $m\in\N$ and let $n$ be the exponent of the largest power of a prime $p$ such that $p^n$ divides $m$.
 Then $\A := (\Z_m,x-y+z)$ is polynomially equivalent to $\A_0 := (\Z_m,+)$, the regular module of the ring
 $\R:=(\Z_m,+,\cdot,1)$. Clearly $\A_0$ is distributive and the Jacobson radical $J(\R)$ satisfies $J(\R)^n=0$.
 Let $\B$ be a finite module of order coprime to $m$.
 By Theorem~\ref{thm:distributive} every clonoid from $\A_0$ into $\B$ is generated by $n$-ary functions and
 every clonoid from $\A$ to $\B$ is generated by $n+1$-ary functions.

 We first prove Theorem~\ref{thm:distributive} for modules $\A,\B$ in Section~\ref{sec:distributive}.
 The key observation for this is that under the given assumptions, for every $k\in\N$ we can find a specific linear
 equation that is satisfied by every $k$-ary function from $\A$ to $\B$ (Theorem~\ref{thm:Zn}).
 As it turns out, we can uniformly interpolate $k$-ary functions from $\A$ to $\B$ by their $n$-ary $\A,\B$-minors
 (see Sections~\ref{sec:general},~\ref{sec:uniform}). Our proofs for these interpolation results make essential use
 of the structure of the distributive module $\A$ (see Section~\ref{sec:dstructure}).
 
 In Section~\ref{sec:abelian} we then extend these results from modules to algebras that are polynomially equivalent to
 modules. The main difference there is that if $\A$ does not have a constant term function $0$, we have to replace
 that by a projection, i.e., an additional variable $z$, in our interpolation arguments. This additional variable
 yields that the clonoids in the general case are generated by $n+1$-ary functions instead of $n$-ary as for modules.
 
 Finally we investigate necessary conditions such that every clonoid from $\A$ to $\B$ is generated by $n$-ary
 functions in Section~\ref{sec:number}. We show that every subalgebra of $\A$ has to be generated by $n$ elements in
 Lemma~\ref{noncyclic}.
 Further, if every clonoid from a regular $\R$-module $\A$ to $\B$ is generated by unary functions, then $\R$ is semisimple
 (Lemma~\ref{JacobsonRadical}).
 Hence for a ring $\R$ with $J(\R)\neq 0$ and $J(\R)^2=0$, the assertion of Theorem~\ref{thm:distributive}
 that all clonoids from distributive $\R$-modules $\A$ to coprime modules $\B$ are generated by binary functions
 cannot be strengthened to unary functions.

 As consequence and partial converse to Theorem~\ref{thm:distributive} we can characterize the pairs of finite modules
 $\A$ over commutative rings and modules $\B$ such that every clonoid from $\A$ to $\B$ is generated by unary functions.

\begin{theorem}\label{CommutativeSummary}
 For a finite faithful $\R$-module $\A$ over a commutative ring $\R$ and a finite $\S$-module $\B$ the following are equivalent:
\begin{enumerate}
\item Every clonoid from $\A$ to $\B$ is generated by unary functions.
\item \label{it:direct} 
 The orders of $\A$ and $\B$ are coprime,
 $\R$ is a direct product of finite fields and $\A$ is isomorphic to the regular $\R$-module.
\end{enumerate}  
\end{theorem}

 The proof is given in Section~\ref{sec:number}. Note that item~\eqref{it:direct} in Theorem~\ref{CommutativeSummary}
 implies that $\A$ is a distributive module.
 In particular, for a prime $p$, not every clonoid from the group $(\Z_p^2,+)$ to a module of coprime order
 is generated by unary functions. So Theorem~\ref{thm:distributive} does not generalize to non-distributive
 modules $\A$.
 
 In general it remains open whether the condition that $\A$ is distributive is necessary so that the number of clonoids
 from $\A$ into any coprime module is finite.
 We do not even know the answer to the following.

\begin{que}
 For a prime $p$, how many clonoids are there from the group $(\Z_p^2,+)$ into a finite module of order coprime to $p$?  
\end{que}


\section{Preliminaries and Notation} \label{Prelims}

\subsection{General} \label{sec:general}
 We write $\N := \{1,2,\dots\}$ for the set of positive integers and $[n] := \{1,\dots,n\}$ for $n\in\N$.

 Algebras are pairs $\A := (A,F)$ with \emph{universe} $A$ and a set $F$ of finitary operations on $A$
 (the \emph{basic operations} of $\A$).
 
\begin{definition}
  Let $\A,\B$ be algebras with universes $A,B$, respectively.
\begin{enumerate}
\item
 Let $F(A,B) := \bigcup_{k\in\N} B^{A^k}$ denote the clonoid of finitary functions from $\A$ to $\B$.
\item  
 For $F \subseteq F(A,B)$, we denote the smallest clonoid from $\A$ to $\B$ that contains $F$ by $\langle F \rangle_{\A, \B}$
 and call it the clonoid \emph{generated} by $F$.
 For $f\in F(A,B)$, we also write $\langle f\rangle_{\A,\B}$ instead of $\langle \{ f \} \rangle_{\A, \B}$.
\item
 A clonoid $C$ from $\A$ to $\B$ is \emph{finitely generated} if it is generated by some finite set $F \subseteq F(A,B)$
\item
 If $g\in\langle f\rangle_{\A,\B}$ for  $f\in F(A,B)$, we call $g$ an $\A,\B$-\emph{minor} of $f$.
\end{enumerate} 
 If $\A$ and $\B$ are clear from context, we also write $\langle F \rangle$ and $\langle f\rangle$ instead of
 $\langle F \rangle_{\A, \B}$ and $\langle f\rangle_{\A,\B}$.
\end{definition}

\begin{definition}
 For $C \subseteq F(A,B)$ and $k \in \N$, let  $C^{(k)} := C \cap B^{A^k}$ denote the set of $k$-ary functions in $C$.
\end{definition}

\subsection{The lattice of clonoids}

 For algebras $\A$ and $\B$, we denote the set of all clonoids from $\A$ to $\B$ by $\C_{\A,\B}$.
 Then $\C_{\A,\B}$ forms a lattice with meet given by intersection and the join of clonoids $C,D$ given by the clonoid generated
 by the union $\langle C \cup D \rangle_{\A,\B}$.
 This is a bounded lattice with bottom element $\emptyset$ and top element $F(A,B)$.

 The join of clonoids $C,D$ from $\A$ to a module $\B$ is just the pointwise sum of the functions in $C$ and $D$,
\[ C+D := \{f+g \st k\in\N, f\in C^{(k)}, g\in D^{(k)} \}. \]

 We note some straightforward relations for clonoids between algebras, their expansions, quotients and subalgebras.

\begin{lemma} \label{lem:expansions}
 Let $\A,\B$ be algebras with expansions $\A^+,\B^+$, respectively. Then $\C_{\A^+,\B^+}$ is a meet-subsemilattice of $\C_{\A,\B}$.
\end{lemma}

\begin{proof}
 Every clonoid from $\A^+$ to $\B^+$ is also a clonoid from $\A$ to $\B$. Clearly the intersection of clonoids in $\C_{\A^+,\B^+}$ is
 the same as in $\C_{\A,\B}$. Since joins in  $\C_{\A^+,\B^+}$ require closure under $\Clo(\A^+)$ and $\Clo(\B^+)$, they may properly contain
 the joins in $\C_{\A,\B}$. 
\end{proof}

 For any equivalence relation $\alpha$ on $A$, any function $f\colon (A/\alpha)^k \to B$ can be lifted to a function
 $f'\colon A^k \to B$ that is constant on $\alpha$-blocks via
 \[f'(x_1,\ldots,x_k) := f(x_1/\alpha, \ldots, x_k/\alpha).\]
 
\begin{lemma} \label{quotients}
 Let $\A$ and $\B$ be algebras, let $\alpha$ be a congruence on $\A$ and let $\B'$ be a subalgebra of $\B$.  Then
\[ \varphi\colon \C_{\A/\alpha, \B'} \to \C_{\A,\B}, \ C \mapsto \{f' \st f \in C \}, \]
 is a lattice embedding.
\end{lemma}

\begin{proof}
 Let $C \in \C_{\A/\alpha, \B'}$. To show $\varphi(C) \in \C_{\A, \B}$, let $f \in C^{(k)}$,
 $s_1, \ldots, s_k \in \text{Clo}(\A)^{(m)}$, and $(x_1,\dots,x_m)\in A^m$.
 Then
 \begin{align*}
  [f'(s_1, \ldots, s_k)](x_1, \ldots,x_m) & = f'(s_1(x_1, \ldots, x_m), \ldots, s_k(x_1, \ldots, x_m)) \\
                                         & = f(s_1(x_1, \ldots,x_m)/\alpha, \ldots, s_k(x_1, \ldots, x_m)/\alpha)  \\
                                         & = f(s_1(x_1/\alpha,\ldots,x_m/\alpha), \ldots, s_k(x_1/\alpha, \ldots, x_m/\alpha)) \\
   & = [f(s_1, \ldots, s_k)]' (x_1,\ldots,x_m) \in \varphi(C).
\end{align*}
 Also for $f_1, \ldots, f_k \in C^{(m)}$, $t \in \Clo(\B)^{(k)}$ and $(x_1,\dots,x_m)\in A^m$, 
 \begin{align*}
   [t(f_1', \ldots, f_k')](x_1, \ldots, x_m) & = t(f_1'(x_1, \ldots, x_m),\ldots, f_k'(x_1, \ldots, x_m)) \\
                                           & = t(f_1(x_1/\alpha, \ldots, x_m/\alpha), \ldots, f_k(x_1/\alpha, \ldots, x_m/\alpha)) \\
   & = [t(f_1,\ldots,f_k)]'(x_1, \ldots, x_m) \in \varphi(C).
\end{align*} 
 Hence $\varphi(C)$ is a clonoid from $\A$ to $\B$. 

 That $\varphi$ is injective and preserves meets and joins is immediate.
\end{proof}

\begin{lemma} \label{lem:subalgebras}
 Let  $\A$ and $\B$ be algebras and let $\A'$ be a subalgebra of $\A$.
 Then restriction of functions from $A$ to $A'$ yields a lattice homomorphism from $\C_{\A,\B}$ to $\C_{\A',\B}$.
\end{lemma}

\begin{proof}
  Straightforward since term functions of $\A'$ are restrictions of term functions of $\A$. 
\end{proof} 

 We give one easy example of how clonoids arise in the description of expansions of algebras.
 Let $\A$ be a module with a submodule $B$. For $f\colon A^k\to A$ and $g\colon (A/B)^k\to B$ write
\[ f+g'\colon A^k\to A,\ x \mapsto f(x)+g(x+B^k). \]
 For a clonoid $C$ from $\A/B$ to $\B$ let 
\[ \Clo(\A)+C := \{f+g' \st k\in\N, f\in\Clo(\A)^{(k)}, g\in C^{(k)} \}. \]
 
\begin{lemma} \label{lem:ABexpansion}
  Let $\A$ be a module with a submodule $\B$,
 let $\C$ be the lattice of clonoids from $\A/B$ to $\B$ that contain all additive functions,
 and let $\mathcal{D}$ be the lattice of clones on $A$ that contain $\Clo(\A)$.
 Then 
\[ \varphi\colon \C\to \mathcal{D},\ C \mapsto \Clo(\A)+ C, \]
 is a lattice embedding.
\end{lemma}

\begin{proof}
 Let $C$ be a clonoid from $\A/B$ to $\B$. Since $C$ contains the constant $0$ function,
 $\Clo(\A)\subseteq\Clo(\A)+C$ and the latter contains all projections on $A$.
 To see that $\Clo(\A)+C$ is closed under composition, let $u\in\Clo(\A),v\in C$ be $n$-ary and let
 $f_1,\dots,f_n\in\Clo(\A),g_1,\dots,g_n\in C$ be $k$-ary. Then
\begin{align*}
  & \left[u+v'\right](f_1+g'_1,\dots,f_n+g'_n)  \\
  = & u(f_1+g'_1,\dots,f_n+g'_n) + v'(f_1+g'_1,\dots,f_n+g'_n) \\
  = & u(f_1,\dots,f_n) + [u(g_1,\dots,g_n)]'+ [v(f_1,\dots,f_n)]' 
\end{align*}
since $u$ is linear on $\A$ and $v'$ is constant on $B$-cosets.
 Note that $u(f_1,\dots,f_n)\in\Clo(\A)$ and $u(g_1,\dots,g_n)+v(f_1,\dots,f_n)\in C$ by the respective
 closure properties.
 Hence $\left[u+v'\right](f_1+g'_1,\dots,f_n+g'_n)$ is in $\Clo(\A)+C$ and the latter is a clone.

 To see that $\varphi$ is injective, let $C_1,C_2\in\mathcal{C}$ such that $\Clo(\A)+C_1 = \Clo(\A)+C_2$.
 Then for every $f_1\in C_1$ there exist $t\in\Clo(\A)$ and $f_2\in C_2$ such that $f'_1 = t+f'_2$.
 Since $f'_1-f'_2 = t\in\Clo(\A)$ is additive, also the induced function $f_1-f_2$ from $\A/B$ to $B$ is additive.
 By the assumption that $C_2$ contains all additive functions as well as $f_2$, we obtain $(f_1-f_2) + f_2 = f_1 \in C_2$.
 Hence $C_1 \subseteq C_2$. The converse inclusion follows by symmetry. Thus $\varphi$ is injective.

 That $\varphi$ preserves meets and joins is straightforward.
\end{proof}

 We note that all the expansions $\A_C := (A, \Clo(\A)+ C)$ in Lemma~\ref{lem:ABexpansion}
 are $2$-nilpotent with respect to the term condition commutator~\cite[Chapter 7]{FM:CTC}.
 In particular the congruence modulo $B$ is central in $\A_C$.

\subsection{Relational description of clonoids}  \label{sec:Galois}
 Pippenger~\cite{Pi:GTMF} extended the Galois theory between clones and relations to the setting of clonoids between sets
 and pairs of relations. Couceiro and Foldes~\cite{CF:FCRC} further expanded this to our setting of clonoids between algebras.

\begin{definition}
 Let $A$ and $B$ be  nonempty sets and let $n \in \N$.  For $R \subseteq A^n, S \subseteq B^n$, let
\[ \Pol(R,S) := \bigcup_{k\in\N} \{f\colon A^k \to B \mid f(R,\dots,R) \subseteq S \} \]
 denote the set of \emph{polymorphisms} of the relational pair $(R,S)$. 
\end{definition}
 
\begin{theorem}\cite[cf.~Theorem 2]{CF:FCRC} Let $\A$ and $\B$ be algebras with $|A|$ finite.
 For $C \subseteq \bigcup_{n \in \N} B^{A^n}$ the following are equivalent:
\begin{enumerate}
\item $C$ is a clonoid from $\A$ to $\B$.
\item

 There exist subalgebras $R_i,S_i$ of $\A^{m_i}, \B^{m_i}$, respectively, for $m_i\in\N$ and $i$ in a set $I$ such that
  $C = \bigcap_{i \in I} \Pol(R_i,S_i)$.

 \end{enumerate}
\end{theorem}

For example, for modules $\A$ and $\B$, the additive functions from $\A$ to $\B$ form a clonoid, 
since $f \colon A^k \rightarrow B$ is additive if and only if $f \in \Pol(R,S)$ where 
\[ R := \{ (x,y,z) \in A^3 \st x+y =z \} \le \A^3, \]
and
\[ S := \{(x,y,z) \in B^3 \st x+y = z \} \le \B^3. \]

 In \cite{AM:FGEC} Aichinger and the first author of this paper showed that in certain settings
 the number of relational pairs needed to determine a clonoid is finite.
 
 A finite algebra $\B$ has \emph{few subpowers} if there exists a polynomial $p$ such that for every $n\in\N$
 the number of subalgebras of $\B^n$ is at most $2^{p(n)}$.
 All finite algebras with (quasi)group or lattice operations have few subpowers. Hence the following holds in particular for
 $\B$ a group or module.

\begin{theorem}\cite[Theorem 5.3]{AM:FGEC} \label{thm:cube}
 Let $\A,\B$ be finite algebras such that $\B$ has few subpowers.
\begin{enumerate}  
\item Then the lattice of clonoids from $\A$ to $\B$ satisfies the descending chain condition.
\item Every clonoid from $\A$ to $\B$ is finitely related (i.e.
  the polymorphism clonoid of a single relational pair). 
\item The number of clonoids from $\A$ to $\B$ is finite or countably infinite.
\end{enumerate}  
\end{theorem}

We note that between finite algebras there are only finitely many clonoids if and only if they are all generated by functions
of some bounded arity.

\begin{lemma} \label{lem:CABfinite}
 For finite algebras $\A,\B$ the following are equivalent:
\begin{enumerate}
\item
 The number of clonoids from $\A$ to $\B$ is finite.
\item
 There exists some $n\in\N$ such that every clonoid from $\A$ to $\B$ is generated by $n$-ary functions.
\item
 There exists some $n\in\N$ such that for all $k\in\N$ every function $f\colon A^k\to B$ is generated by
 its $n$-ary $\A,\B$-minors.
\end{enumerate}  
\end{lemma}

\begin{proof}
 Straightforward. 
\end{proof}

\subsection{Abelian Mal'cev algebras} \label{sec:Malcev}
 Algebras $(A,F_1)$ and $(A,F_2)$ on the same universe are \emph{polynomially equivalent} if they have the same clone
 of polynomial functions.

 A ternary operation $m$ on a set $A$ is \emph{Mal'cev} if 
\[ m(x,y,y) = x = m(y,y,x) \text{ for all } x,y\in A. \]
 An algebra $\A$ is \emph{Mal'cev} if it has a term operation which is Mal'cev.

 Every expansion of a group $(A,+,-,0)$ has a Mal'cev term operation $m(x,y,z) = x-y+z$.
 A Mal'cev algebra is abelian (with respect to the term condition commutator) if and only if it is polynomially equivalent to
 a module~\cite{FM:CTC}. We will need the following explicit characterization of abelian Mal'cev algebras in
 Section~\ref{sec:abelian}.

\begin{lemma} \cite[Chapters 5, 9]{FM:CTC} \label{lem:abelian}
 Let $\A$ be an abelian algebra with Mal'cev term operation $m$. Then
\begin{enumerate}
\item
 the set of rank-$1$-commutator terms $R_\A := \{ r\in \Clo(\A)^{(2)} \st r(z,z)=z\ \forall z\in A \}$ of $\A$ forms a ring
 $\R_\A$ under
\begin{align*}
 & r(x,z) + s(x,z) := m(r(x,z),z,s(x,z)), \\
 &-r(x,z) := m(z,r(x,z),z) \\ 
 & r(x,z) \cdot s(x,z) := r(s(x,z),z)   
\end{align*}
 for $r,s\in R$ and with neutral elements $z,x$ for $+,\cdot$, respectively. 
\item \label{it:poly}
 Fixing any $0\in A$ as neutral element, $A$ forms an $\R_\A$-module $\A_0$   
 under
\begin{align*}
 & a +_0 b := m(a,0,b), \quad -_0 a := m(0,a,0), \\ 
 & r*_0 a := r(a,0)
\end{align*}
 for $a,b\in A, r\in R_\A$.

 The expansion of $\A$ with the constant operation $0$ is term equivalent to the expansion of $\A_0$ with $\Clo(\A)^{(1)}$.
 In particular, $\A$ and $\A_0$ are polynomially equivalent.
\item \label{it:iso}
 For $0,z\in A$,
\[ h\colon A\to A,\ x \mapsto m(x,0,z), \]
 is an $\R_\A$-module isomorphism from $\A_0$ to $\A_z$. 

\end{enumerate}   
\end{lemma}

 Let $\A$ be an abelian Mal'cev algebra.
 For $r =(r_{ij})_{i,j\in [k]}$ a $k\times k$-matrix over $R_\A$, $x= (x_1,\dots,x_k)$ a $k$-tuple of variables and
 $z$ another variable, define
\[ r*_zx := \left(\sum_{j=1}^k r_{ij}(x_j,z)\right)_{i\in [k]} \]
 with term functions added by $s(x,z)+_z t(x,z) := m(s(x,z),z,t(x,z))$.
 Then $r*_zx$ is a $k$-tuple of $k+1$-ary term functions over $\A$.

 For each abelian Mal'cev algebra $\A$, the ring $\R_\A$ acts faithfully on $\A$. Hence the nilpotence degree of the Jacobson
 radical of $\R_\A$ is at most the height of the congruence lattice of $\A$.

\subsection{Distributive modules} \label{sec:dstructure}

 Recall that a ring is \emph{local} if it has a unique maximal left ideal. A ring $\R$ is \emph{semiperfect} if it has a set $e_1, ..., e_\ell$ of orthogonal idempotents with
 $e_1+\dots+e_\ell = 1$ such that $e_i R e_i$ is a local ring for all $i\in [\ell]$. Then the idempotents
 $e_1,\dots,e_\ell$ are called local.
 All left Artinian (in particular all finite) rings are semiperfect.

 A module is \emph{uniserial} if its submodules are linearly ordered by inclusion. 
 In particular, simple modules and $\Z_{p^n}$ for any prime $p$ and $n\in\N$ are uniserial.
 For an example over a noncommutative ring, let $F$ be a field and 
\[ R := \left\{ \begin{pmatrix} a & b \\ 0 & c \end{pmatrix} \st a,b,c\in F \right\}. \]
 Then $A := F^2$ with the usual action of matrices on column vectors forms a uniserial $\R$-module with submodules
 $0\leq F\times 0 \leq F^2$.

 Distributive modules over semiperfect rings can be decomposed as direct sums of uniserial modules over local subrings as follows.

\begin{lemma} \label{lem:distributive}
 Let $\R$ be a semiperfect ring with $1=\sum_{i=1}^\ell e_i$ for local orthogonal idempotents $e_1,\dots,e_\ell$ and let
 $\A$ be a distributive $\R$-module. Then
\begin{enumerate}
\item \cite[Lemma 4]{Fu:RLIM} \cite[1.28]{Tu:SSM} \label{it:eiA}
 $e_iA$ is a uniserial $e_iRe_i$-module for all $i\in [\ell]$.
\item \label{it:R'}
 $A= \sum_{i=1}^\ell e_i A$ is a distributive module over the subring $R' := \sum_{i=1}^\ell e_iRe_i$ of $\R$ with the
 induced action from $\R$. 
\item \label{it:JR'}
 $J(\R') \subseteq J(\R)$.
\end{enumerate}
\end{lemma}

\begin{proof}
\eqref{it:eiA}
 We include the short proof from~\cite[Lemma 4]{Fu:RLIM} for the convenience of the reader.
 For $i \in [\ell]$, write $e := e_i$ and $J := J(\R)$.
 We first show that the set of $\R$-modules $\{ Rex \st x \in A \}$ is linearly ordered under inclusion.
 Seeking a contradiction, suppose that there are $x, y \in A$ such that $Rex$ and $Rey$ are not comparable.
 Since $Je$ is the unique maximal $\R$-submodule of $Re$ by the assumptions,
 $Rex$ and $Rey$ have unique maximal submodules $Jex$ and $Jey$, respectively.
 Hence $Rex \cap Rey = Jex \cap Jey$. It follows that
\[\varphi \colon Rex / Jex \times  Rey / Jey \rightarrow (Rex + Rey)/(Jex + Jey), \]
\[ (a+Jex, b+Jey) \mapsto a+b+ (Jex+Jey)\] 
is an isomorphism. That $\varphi$ is a well-defined surjective homomorphism is clear.
 For injectivity, let $a\in Rex, b\in Rey$ such that $a+b \in Jex+Jey$. Then $a+b = c+d$ for some $c \in Jex, d \in Jey$.
 So $a-c = d-b\in Rex\cap Rey$. Since $Rex \cap Rey = Jex \cap Jey$, it follows that $a-c \in Jex$ and consequently
 $a\in Jex$.
 Similarly, $b \in Jey$, and so $\varphi$ is injective.
Since $Rex / Jex \times  Rey / Jey \cong (Re/Je)^2$ is not distributive, this contradicts the distributivity of $\A$.
Thus $\{ Rex \st x \in A \}$ is linearly ordered.

 So for each $x, y \in A$ either $ex \in eRey$ or $ey \in eRex$. Therefore $eA$ is a uniserial $eRe$-module.

\eqref{it:R'} Using the orthogonality of the idempotents, we see that every $\R'$-submodule of $\A$ is a direct sum of
 $\R'$-submodules of $e_iA$ for $i\in [\ell]$. Further the $\R'$-submodules of $e_iA$ coincide with the
 $e_iRe_i$-submodules. Since the latter are linearly ordered by~\eqref{it:eiA}, it follows that $\sum_{i=1}^\ell e_i A$
 is a distributive $\R'$-module.  

 For \eqref{it:JR'}, let $s\in J(e_1Re_1)$ and $r\in R$. Then $s = e_1 s e_1$ and so
 \[1-rs = 1 - \sum_{i=1}^\ell e_i r e_1 s e_1 = 1 - e_1 r e_1 s e_1 - \sum_{i > 1} e_i r e_1 s e_1.\]
 
 Since $e_1 r e_1 s e_1$ is in the Jacobson radical of the local ring $e_1 R e_1$ with identity $e_1$,
 we see that $e_1 - e_1 r e_1 s e_1$ is a unit in $e_iRe_i$. Since the quotient of $e_1 R e_1$ by $J(e_1Re_1)$
 is a division ring, the inverse of $e_1 - e_1 r e_1 s e_1$ has the form $e_1-u$ for some $u \in J(e_1 R e_1)$. 
 By the orthogonality of idempotents again,
\begin{align*}
(1-u)(1-rs) & = \underbrace{(1-u)(1- e_1 r e_1 s e_1)}_{=1} - (1-\underbrace{u}_{=ue_1})(\sum_{i>1} e_i r e_1 s e_1) \\ 
& = 1 - \sum_{i > 1} e_i r e_1 s e_1.
\end{align*}

Since $(\sum_{i > 1} e_i r e_1 s e_1)^2 = 0$, we see that $1 - \sum_{i > 1} e_i r e_1 s e_1$ has inverse $1 + \sum_{i > 1} e_i r e_1 s e_1.$ Hence $1 -rs$ is a unit for every $r \in R$, so $s \in J(\R)$. 

A similar argument shows that $J(e_i R e_i) \subseteq J(\R)$ for each $i \in [\ell]$.
Hence $J(\R') = \sum_{i=1}^\ell J(e_i Re_i) \subseteq J(\R)$.
\end{proof}

 Let $\A$ be a module over a ring $\R$, let $k\in\N$ and $T = (t_{ij})_{1\leq i,j \leq k}$
 a $k\times k$-matrix over $\R$. Then $T$ acts on $x := (x_1,\dots,x_k)$ in $\A^k$ by 
\[ Tx := (\sum_{j=1}^k t_{1j}x_j, \dots, \sum_{j=1}^k t_{kj}x_j). \]
 Here $A^k\to A^k,\ x\mapsto Tx,$ can be considered as a $k$-tuple of $k$-ary term functions of $\A$
 but is not necessarily an $\R$-homomorphism if $\R$ is noncommutative.

 For interpolating functions with domain $\A^k$ in Section~\ref{sec:distributive}, it will be useful to cover $\A^k$ by its
 $\R$-submodules that are isomorphic to $\A\times (J(\R)A)^{k-1}$.
 We collect some information about these submodules.
 Recall that by Lemma~\ref{lem:distributive} every finite distributive module has a reduct $\A$ as in the following.
 
\begin{lemma} \label{lem:free}
 Let $\R$ be a finite direct product of finite local rings with $J := J(\R)$,
 let $\A$ be a finite distributive $\R$-module, let $k\in\N$ and let
  \[ V := \{ M\leq \A^k \st (JA)^k\leq M, M/(JA)^k\cong \A/JA\}. \]
 Then
\begin{enumerate}
\item \label{it:transitive}
 for every $M\leq \A^k$ such that $(JA)^k\leq M$ and $M/(JA)^k$ embeds into $\A/JA$
 there exists an invertible $k\times k$-matrix $T$ over $\R$ such that $TM\leq A\times (JA)^{k-1}$;  
 in particular, $\GL_k(\R)$ acts transitively on $V$. 
\item \label{it:cap}
 For all distinct $M,N\in V$ there exists $L<\A$ such that $M\cap N \leq L^k$. 
\item \label{it:cup}
 $A^k = \bigcup V \cup \bigcup \{ L^k \st L<\A \}$.
\end{enumerate} 
\end{lemma}  

\begin{proof}
 Assume $\R = \sum_{i=1}^{\ell} \R_i$ for finite local rings $\R_i$ for $i\in [\ell]$.
 By Lemma~\ref{lem:distributive} we have a direct decomposition $A = \sum_{i=1}^{\ell} A_i$ into uniserial $\R$-submodules
 $A_i$ for $i\in [\ell]$.
 Factoring $\R_i$ by the annihilator of $\A_i$ if necessary, we may assume that $\R_i$ is faithful on $\A_i$ for all $i\in [\ell]$.
 Since $\R_i$ is local and finite, $\R_i/J(\R_i)$ is a finite field $K_i$ for each $i\in[\ell]$. 
 Further $\overline{\R} := \R/J \cong \sum_{i=1}^\ell K_i$ and $\overline{\A} := \A/JA$ is isomorphic to the
 regular $\overline{\R}$-module $\sum_{i=1}^\ell K_i$.
 
 For proving~\eqref{it:transitive}, let $M\leq \A^k$ such that $(JA)^k\leq M$ and $\overline{M} := M/(JA)^k$
 embeds into $\overline{\A}$. Then we have $S\subseteq [\ell]$ such that $\overline{M}\cong  \sum_{i\in S} K_i$.
 More precisely, we also have $j_i\in [k]$ for $i\in S$ such that 
 \[ \overline{M} = \sum_{i\in S} 0 \times \dots \times 0 \times \underset{j_i}{K_i} \times 0 \times \dots \times 0 \leq \overline{\A}^k. \]
 Clearly there exists an invertible $k\times k$ matrix $\overline{T}$ over $\overline{\R}$ that moves $K_i$ from the
 $j_i$-th component of $\overline{A}^k$ to the first component for all $i\in S$.
 Hence $\overline{T}\overline{M} = \sum_{i\in S} K_i \times 0^{k-1} \leq \overline{\A} \times 0^{k-1}$.

 Now let $T\in R^{k\times k}$ such that the projection of $T$ modulo $J$ is $\overline{T}$.
 Since $\overline{T}$ is invertible over $\overline{\R}$, we have $S\in R^{k\times k}$ such that $ST = I_k-U$ for $I_k$
 the $k\times k$ identity matrix over $\R$ and some $U\in J^{k\times k}$.
 Since $J$ is nilpotent, it follows that $ST\in\GL_k(\R)$ and further $T\in\GL_k(\R)$.
 Since $\overline{T}\overline{M} = N/J \times 0^{k-1}$ for some $JA\leq N\leq\A$, 
 we obtain $TM \subseteq N \times (JA)^{k-1}$.
 Equality follows since $|M| = |N\times (JA)^{k-1}|$ and $T$ is invertible. Thus~\eqref{it:transitive} is proved.
 
 For~\eqref{it:cap} let $M,N\in V$ be distinct. By~\eqref{it:transitive}
 we have $T\in\GL_k(\R)$ and $JA\leq L<\A$ such that $T(M\cap N) = L\times (JA)^{k-1}$.
 Hence $M\cap N \leq T^{-1} L^k \leq L^k$.

 For~\eqref{it:cup} let $a\in A^k$ and $M := Ra+(JA)^k$.
 Then $\overline{M} := M/(JA)^k$ is a cyclic $\overline{\R}$-module, hence isomorphic to a quotient of the regular
 $\overline{\R}$-module $\overline{\A}$.
 Since $\overline{\A}$ is semisimple, $\overline{M}$ also embeds into $\overline{\A}$.
 By~\eqref{it:transitive} we have $T\in\GL_k(\R)$ and $JA\leq L\leq\A$
 such that $TM = L\times (JA)^{k-1}$.
 If $L=A$, then $M\in V$ and $a\in\bigcup V$. Else if $L < \A$, then $M \leq T^{-1} L^k \leq L^k$ and $a\in\bigcup\{ L^k \st L<\A\}$.
 Thus $A^k \subseteq \bigcup V \cup \bigcup \{ L^k \st L<\A \}$. The converse inclusion is trivial.
\end{proof}

 \subsection{Inner rank of matrices}

 The rank of matrices over fields can be generalized to arbitrary rings as follows (see~\cite[Proposition 5.4.3]{Co:FIR}).
 Let $\R$ be a ring, $m,n\in\N$ and $A \in R^{m\times n}$. The \emph{inner rank} of $A \neq 0$ is defined as
 \[ \rk(A) := \min \{ r\in\N \st A = BC \text{ for some } B\in R^{m\times r}, C\in R^{r\times n} \} \]
 and $\rk(0):= 0$.
 
 Then $\rk(A)$ is the least $r$ such that the right $\R$-submodule of $R^m$ that is generated by the columns of $A$ is
 contained in an $r$-generated module (equivalently, the least $r$ such that the left $\R$-submodule of $R^n$ that is generated
 by the rows of $A$  is contained in an $r$-generated module). Note that in particular $\rk(A) \leq\min(m,n)$.

 If $\R$ is a field, the inner rank is just the usual row or column rank.

\subsection{Uniformly generated functions} \label{sec:uniform}

Let $\A,\B$ be algebras.
 For $k,n\in\N$, let 
\begin{multline*}
R_{k,n} := \{ r\in (\Clo(\A)^{(k)})^k \st \\ r(x) = ( v_1 (w_1(x),\dots,w_n(x)), \dots, v_k (w_1(x),\dots,w_n(x)) )  \\ \text{ for some }
   v_1,\dots,v_k\in\Clo(\A)^{(n)}, w_1,\dots, w_n\in\Clo(\A)^{(k)} \} \end{multline*}
 denote the set of $k$-tuples of $k$-ary term functions on $\A$ that factor through $n$-ary term functions.

 For an $\R$-module $\A$ we simply have
\[  R_{k,n} = \{ A^k \to A^k, x\mapsto ax \st a\in R^{k\times k}, \rk(a) \leq n \}. \] 
 A function $f\colon A^k \rightarrow B$ is \emph{generated by its $n$-ary $\A,\B$-minors} if $f$ is in the
 clonoid from $\A$ to $\B$ that is generated by the $n$-ary $\A,\B$-minors of $f$.
 More explicitly $f$ is generated by its $n$-ary $\A,\B$-minors iff there exist
 $\ell\in \N, r_1, \ldots, r_\ell \in R_{k,n}$ and $s \in \Clo(\B)^{(\ell)}$ such that 
\begin{equation} \label{eq-nary-generation}
 f(x) = s(f(r_1(x)), \ldots, f(r_\ell(x))) \text{ for all } x \in A^k.
\end{equation}
 We say that a set of functions $U\subseteq F(A,B)^{(k)}$ is \emph{uniformly generated by $n$-ary
 $\A,\B$-minors} if there exist $\ell\in\N, r_1, \ldots, r_\ell \in R_{k,n}$ and $s \in \Clo(\B)^{(\ell)}$ such that
 \eqref{eq-nary-generation} holds for all $f \in U$. 

 Moreover, we say that an operator $d \colon F(A,B)^{(k)} \rightarrow F(A,B)^{(k)}$ \emph{can be uniformly represented by
 $n$-ary $\A,\B$-minors on $U$} if there exist $\ell, r_1, \ldots, r_\ell, s$ as above such that for all $f \in U$, 
\[d(f)(x) = s(f(r_1(x)), \ldots, f(r_\ell(x))) \text{ for all } x \in A^k.\]

 For example, for modules $\A,\B$ every binary affine function $f\colon A^2\to B$ satisfies
\[ f(x_1,x_2) = f(x_1,0)-f(0,0)+f(0,x_2) \text{ for all } x_1,x_2\in A, \]
 hence is generated by its unary $\A,\B$-minors $f(x,0),f(0,0),f(0,x)$.
 Moreover,
 the set of affine functions $U \subseteq F(A,B)^{(2)}$ is uniformly generated by unary $\A,\B$-minors.
 Thus the identity operator can be uniformly represented by unary $\A,\B$-minors on $U$ (but in general
 not on all of $F(A,B)^{(2)}$).

 We will use the following properties of uniformly generated sets of functions between modules in
 Section~\ref{sec:distributive}. 

\begin{lemma} \label{lem:uniform}
 Let $\A$ be an $\R$-module, $\B$ an $\S$-module, $k,n\in\N$, $U\subseteq F(A,B)^{(k)}$ and
 $d: U\to F(A,B)^{(k)},\ f\mapsto f'$.
\begin{enumerate}
\item \label{it:uniform}
 Then $d$ can be uniformly represented by $n$-ary $\A,\B$-minors on $U$ if and only if there exists
 $s\colon \{r\in R^{k\times k} \st \rk(r) \leq n \} \to S$ with finite support such that for all $f\in U$ and all
 $x\in A^k$
\[ f'(x) = \sum_{r\in R^{k\times k}, \rk(r)\leq n} s(r) f(rx). \]
\item \label{it:f-f'}
 Assume that $d$ can be uniformly represented by $n$-ary $\A,\B$-minors on $U$ and that
 $\{f-d(f) \colon f \in U\}$ is uniformly generated by $n$-ary $\A,\B$-minors.
 Then $U$ is uniformly generated by $n$-ary $\A,\B$-minors.
\item \label{it:step}
 Let $\ell\in\N$ and let $\M_1,\dots, \M_\ell$ be submodules of $\A$.
 Assume that $F(M_i,B)^{(k)}$ is uniformly generated by $n$-ary $\M_i,\B$-minors for each $i\in[\ell]$ and that
 $$U_0 := \{f\in F(A,B)^{(k)} \st f(M_i^k) = 0 \text{ for all } i\in[\ell] \}$$
 is uniformly generated by $n$-ary $\A,\B$-minors.

 Then $F(A,B)^{(k)}$ is uniformly generated by $n$-ary $\A,\B$-minors.
\end{enumerate}
\end{lemma}

\begin{proof}
 \eqref{it:uniform} is straightforward from the definition.

 For~\eqref{it:f-f'} let $s,t\colon \{r\in R^{k\times k} \st \rk(r) \leq n \} \to S$ such that for all $f\in U$
 and all $x\in A^k$
\begin{equation} \label{eq:f'}
  f'(x) = \sum_{r\in R^{k\times k}, \rk(r)\leq n} s(r) f(rx),
\end{equation} 
\[  (f-f')(x) = \sum_{u\in R^{k\times k}, \rk(u)\leq n} t(u) (f-f')(ux). \]
 Since $\rk(ru) \leq \min(\rk(r),\rk(u))$, it follows that 
\begin{align*}
 (f-f')(x) = & \sum_{u\in R^{k\times k}, \rk(u)\leq n} t(u) \left( f(ux) - \sum_{r\in R^{k\times k}, \rk(r)\leq n} s(r) f(rux) \right)
\end{align*}                            
is generated by $n$-ary $\A,\B$-minors of $f$.
 Together with~\eqref{eq:f'} we see that $f=(f-f') +f'$ is generated by $n$-ary $\A,\B$-minors of $f$
 using the same formula for every $f\in U$. That is, $U$ is uniformly generated  by $n$-ary $\A,\B$-minors.

 For~\eqref{it:step}, let $i\in [\ell]$.
 By the assumptions we have an operator $d_i$ that is uniformly represented by $n$-ary $\A,\B$-minors on
 $F(A,B)^{(k)}$ such that
\[ d_i(f)|_{M_i^k} = f|_{M_i^k} \text{ for every } f\in F(A,B)^{(k)}. \]
 Then $e_i(f) := f-d_i(f)$ is $0$ on $M_i^k$ for all $f\in F(A,B)^{(k)}$.
 Consider 
\[ e_\ell \dots e_2 e_1 (f) = f-\underbrace{d_1(f) - d_2(f-d_1(f)) - \dots}_{=:d(f)}. \] 
 By construction the operator $d$ is uniformly represented by $n$-ary $\A,\B$-minors on
 $F(A,B)^{(k)}$. Further $e_\ell \dots e_2 e_1 (f)$ is $0$ on $M_i^k$ for all $f\in F(A,B)^{(k)}$. Hence
 \[ \{ f-d(f) \st f\in F(A,B)^{(k)} \} \subseteq U_0. \]
 Since $U_0$ is uniformly generated by $n$-ary $\A,\B$-minors by assumption,
 \eqref{it:f-f'} yields that $F(A,B)^{(k)}$ is uniformly generated by $n$-ary $\A,\B$-minors.
\end{proof}

\subsection{Sums} Let $S$ be a subset of an abelian group $(A,+)$. To simplify notation we will also write $\sum S$ for
$\sum_{x\in S} x$.


\section{Clonoids from distributive modules}  \label{sec:distributive}

This section consists of the proof that
 every function from a finite distributive $\R$-module $\A$ into a coprime module $\B$ is 
 generated by its $n$-ary $\A,\B$-minors,
 where $n$ is the nilpotence degree of the Jacobson radical of $\R$.
 This is the basis of all our finite generation results for clonoids between coprime abelian Mal'cev algebras. 
 More precisely we show the following.

\begin{theorem} \label{thm:Zn}
  Let $\A$ be a finite distributive $\R$-module, let $n\in\N$ such that $J(\R)^n = 0$,
  and let $\B$ be an $\S$-module such that $|A|$ is invertible in $\S$.
  
 Then for all $k\in\N$ there exists $s\colon \{r\in R^{k\times k} \st \rk(r) \leq n \} \to S$ such that
 for all $f\colon A^k\to B$ and all $x\in A^k$
\begin{equation*} 
  f(x) = \sum_{r\in R^{k\times k}, \rk(r)\leq n} s(r) f(rx).
\end{equation*}  
\end{theorem}

The equation in Theorem~\ref{thm:Zn} can be viewed as a generalized linearity condition that holds for every $k$-ary
function  $f$ from $A$ to $B$.

\begin{exa}
 We illustrate Theorem~\ref{thm:Zn} for binary functions from $\A=(\Z_2,+)$ to an $\S$-module $\B$ such that $2$ a unit in $\S$.
 Then $F(A,B)^{(2)}$ is uniformly generated by unary $\A,\B$-minors $f(0,0)$, $f(x,0)$, $f(0,x)$, $f(x,x)$
 for $f \in F(A,B)^{(2)}$ via
 \begin{align*} f(x_1,x_2) = & f(0,0) \\
   &+2^{-1} \bigl[  f(x_1,0)+f(x_1+x_2,0)-f(0,0)-f(x_2,0) \\
     & +f(0,x_2)+f(0,x_1+x_2)-f(0,0)-f(0,x_1) \\
  & +f(x_1,x_1)+f(x_2,x_2)-f(0,0)-f(x_1+x_2,x_1+x_2) \bigr].
\end{align*}
 This identity can be derived from the interpolation arguments for Theorem~\ref{thm:Zn} below but is more elementary
 verified by the following case analysis.
 If $x_2=0$, then lines 1 and 2 on the right hand side add up to $f(x_1,0)$ while lines 3 and 4 each cancel.
 If $x_1=0$, then similarly lines 1 and 3 add up to $f(0,x_2)$ while lines 1 and 4 each cancel.
 Finally if $x_1=x_2$, then lines 1 and 4 add up to $f(x_1,x_2)$ while lines 2 and 3 each cancel.
 Hence in each case the right hand side of the formula yields $f(x_1,x_2)$ and the claim is proved.
\end{exa}

 Before proving Theorem~\ref{thm:Zn}, here is a brief outline of our strategy.
 By Lemma~\ref{lem:distributive} it will be enough to consider $\R$ as a direct product of finite local rings.
 We will then use induction on $\A$ to show that $F(A,B)^{(k)}$ is uniformly generated by $n$-ary $\A,\B$-minors. 
 By the induction assumption and Lemma~\ref{lem:uniform}\eqref{it:step} it will suffice to show that
\[ F_0(A,B)^{(k)} := \{f \in F(A,B)^{(k)} \st f(M^k) = 0 \text{ for all } M < \A \} \]
 is uniformly generated by $n$-ary $\A,\B$-minors. For that we will need some auxiliary interpolation results that
 are established in the following two lemmas.

 First we prove that under certain assumptions for any maximal submodule $M<\A$ every $f\in F_0(A,B)^{(k)}$
 generates a function $f'_M$ that is equal to $f$ on $A\times M^{k-1}$ and $0$ else.
 
\begin{lemma} \label{lem:gM}
 Let $\R$ be a finite direct product of finite local rings,
 let $\A$ be a finite distributive $\R$-module, and let $\B$ be an $\S$-module such that $|A|$ is invertible in $\S$.
 For a maximal $\R$-submodule $M$ of $\A$, $k\in\N$
 and $f\in F_0(A,B)^{(k)}$, define
\[f'_M\colon A^k\to B,\ x\mapsto\begin{cases} f(x) & \text{ if } x \in A\times M^{k-1}, \\ 0 & \text{ else}. \end{cases}\] 
 Then the operator $f\mapsto f'_M$ can be uniformly represented by $\A,\B$-minors on $F_0(A,B)^{(k)}$.
\end{lemma}

\begin{proof}
 Assume $\R = \sum_{i=1}^{\ell} \R_i$ for finite local rings $\R_i$ for $i\in [\ell]$.
 By Lemma~\ref{lem:distributive} we have a direct decomposition $A = \sum_{i=1}^{\ell} A_i$ into uniserial $\R$-submodules
 $A_i$ for $i\in [\ell]$.
 Since $M$ is a maximal $\R$-submodule, we may assume that $M = J_1A_1+\sum_{i=2}^\ell A_i$ where $J_1$ is
 the greatest ideal of $\R_1$.
 In particular $A_1$ is nontrivial. It follows that $|A_1|$ and $|R_1|$ are both nontrivial powers of the
 characteristic of the field $\R_1/J$.
 Hence $|R_1|$ is invertible in $\S$ by the assumption on $|A|$.

 We use induction on the number of submodules of $\A_1$.
 
 For the base case, $J_1=0$ and $\A_1$ is a $1$-dimensional vector space over the field $\R_1$.
 Let $e$ be the multiplicative identity of $\sum_{i=2}^{\ell} R_i$.
 Let $f\in F_0(A,B)^{(k)}$.
 We claim that for all $x_1,\dots,x_k\in A$  
\begin{equation} \label{eq:uM1}
 \begin{array}{ll} f'_M(x_1,\dots,x_k) = & |R_1|^{1-k} \left(\sum_{a_2,\dots,a_k\in R_1} f(x_1+\sum_{i=2}^k a_i x_i,x_2,\dots,x_k) \right. \\
    & \left. - \sum_{a_2,\dots,a_k\in R_1} f(ex_1+\sum_{i=2}^k a_i x_i,x_2,\dots,x_k)\right).  \end{array}
\end{equation}
Note that the right hand side is welldefined since $|R_1|$ is invertible in $\S$.
 If $x_2,\dots,x_k\in M$, then~\eqref{eq:uM1} follows since $\sum_{i=2}^k a_i x_i = 0$ and the right hand side simplifies to
\[ f(x_1,x_2,\dots,x_k) - \underbrace{f(ex_1,x_2,\dots,x_k)}_{=0 \text{ since } f(M^k)=0} = f(x_1,x_2,\dots,x_k). \]
Else if $(x_2,\dots,x_k)\not\in M^{k-1}$, then
 the maximality of $M$ in $\A$ implies
\[ \left\{ \sum_{i=2}^k a_i x_i \st a_2,\dots,a_k\in R_1 \right\} = \sum_{i=2}^k R_1 x_i = A_1 \]
 where each element in $A_1$ is attained with the same multiplicity $\frac{|R_1|^{k-1}}{|A_1|}$. 
 So the right hand side of~\eqref{eq:uM1} yields
\[ |R_1|^{-1}\left(\sum f(x_1+A_1,x_2,\dots,x_k) - \sum f(\underbrace{ex_1+A_1}_{=x_1+A_1 \text{ since } A=A_1\times M},x_2,\dots,x_k)\right) = 0. \]
This shows~\eqref{eq:uM1} and the base case of the induction.

 For the induction step let $M = J_1A_1+\sum_{i=2}^\ell A_i$ and $J_1A_1 \neq 0$ in the following.
 Let $N$ be the smallest $\R$-submodule of $\A_1$. Then $N\leq M$ and $N = IA_1$ for $I$ a power of $J_1$.
 In order to apply the induction assumption to $\A/N$, we will take averages of functions over arguments $\sum Ix$. 
 For this we note that by the minimality of $N$,
\begin{equation*} 
  Ix = \begin{cases} 0 & \text{for all } x\in M, \\
  N & \text{for all } x\in A\setminus M. \end{cases}
\end{equation*}
 First we claim that 
\[ \bar{f}(x_1,\dots,x_k) := |I|^{-k^2} \sum_{a_1,\dots,a_k \in I^k} f(x_1 + \sum_{j=1}^k a_{1j} x_j, x_2 + \sum_{j=1}^k a_{2j} x_j,
 \dots, x_k + \sum_{j=1}^k a_{kj} x_j) \]
 satisfies
 \[ \bar{f}(x) = |N|^{-k} \sum f(x+N^k) \text{ for all } x\in A^k. \]
 For $x\in M^k$, this follows from the assumption that $f(M^k) = 0$ and $N\subseteq M$.
 Else if $x\not\in M^k$, it follows from 
\[ \left\{\sum_{j=1}^k a_{ij} x_j \st a_i\in I^k \right\} = \sum_{j=1}^k Ix_j = N, \]
 where each element in $N$ is attained with the same multiplicity $\frac{|I|^{k}}{|N|}$. 
 Since $\bar{f}$ is constant on blocks modulo $N$,
 it naturally induces the map
 \[\bar{\bar{f}} \colon (\A/N)^k \rightarrow \B, \ x+N^k \mapsto \bar{f}(x), \]
 in $F_0(A/N,B)^{(k)}$. 
 By the induction assumption on $\A/N$ with maximal submodule $M/N$, we have 
 $s\colon \{r\in R^{k\times k} \st \rk(r) \leq n \} \to S$ such that for all $g\in F_0(A/N,B)^{(k)}$ and all $x\in A^k$
\[ g'_{M/N}(x+N^k) = \sum_{r\in R^{k\times k}, \rk(r)\leq n} s(r) g(rx+N^k). \]
 In particular, for $f\in F_0(A,B)^{(k)}$, we have
\[ \bar{\bar{f}}'_{M/N}(x+N^k) = \sum_{r\in R^{k\times k}, \rk(r)\leq n} s(r) \bar{\bar{f}}(rx+N^k), \]
 equivalently
\begin{equation} \label{eq:bf'M}
 \bar{f}'_{M}(x) = \sum_{r\in R^{k\times k}, \rk(r)\leq n} s(r) \bar{f}(rx)
\end{equation}
 for all $x\in A^k$.
 By its definition, the operator $f\mapsto\bar{f}$ is uniformly represented by $\A,\B$-minors on $F_0(A,B)^{(k)}$.
 Hence so is $f\mapsto\bar{f}'_{M}$ by~\eqref{eq:bf'M}.
 
 Similarly as for $\bar{f}$ we see that 
\begin{align*} 
\hat{f}(x_1,\dots,x_k) := |I|^{-k(k-1)} \sum_{a_1,\dots,a_k \in I^{k-1}} f(x_1 + \sum_{j=2}^k a_{1,j-1} x_j, x_2 + \sum_{j=2}^k a_{2,j-1} x_j, \\
 \dots, x_k + \sum_{j=2}^k a_{k,j-1} x_j) 
 \end{align*}
 satisfies
 \[ \hat{f}(x) = \begin{cases} f(x)  & \text{if } x\in A\times M^{k-1}, \\
     |N|^{-k} \sum f(x+N^k) = \bar{f}(x) & \text{else}. \end{cases} \]
 Since $f'_M = \hat{f}-(\bar{f}-\bar{f}'_M)$ and the operators $f\mapsto \hat{f}, \bar{f},\bar{f}'_M$, respectively,
 are uniformly represented by $\A,\B$-minors on $F_0(A,B)^{(k)}$, so is $f\mapsto f'_M$.
 Hence the induction step and the lemma is proved. 
\end{proof}  

 Next we show that for $n$ the nilpotence degree of the Jacobson radical $J$ of $\R$, under certain conditions
 $n$-ary $\A,\B$-minors of $f$ generate functions that are equal to $f$ on any submodule
 $N$ isomorphic to $A\times (JA)^{k-1}$ and $0$ else.
 
\begin{lemma} \label{lem:fN}
 Let $\R$ be a finite direct product of finite local rings with $J := J(\R)$ and $n\in\N$ such that $J^n=0$,
 let $\A$ be a finite distributive $\R$-module, let $\B$ be an $\S$-module such that $|A|$ is invertible in $\S$,
 and let $k\in\N$.
 Assume $J=0$ or that $F(JA,B)^{(k-1)}$ is uniformly generated by $n-1$-ary $JA,\B$-minors.
 For $(JA)^k \le N \le \A^k$ such that $N/(JA)^k \cong \A/JA$ and $f\in F_0(A,B)^{(k)}$, let
\[ f_N\colon A^k\to B,\ x\mapsto \begin{cases} f(x) & \text{if } x\in N, \\
    0 & \text{else}. \end{cases} \]
 Then the operator $f\mapsto f_N$ can be uniformly represented by $n$-ary $\A,\B$-minors.
\end{lemma}

\begin{proof}  
 First we show that there exists an operator $f\mapsto \hat{f}$ that is uniformly represented by $n$-ary $\A,\B$-minors
 such that $f$ and $\hat{f}$ are equal on $N = A\times (JA)^{k-1}$.

 If $n=1$ and $J=0$, simply let $\hat{f}(x_1,\dots,x_k) := f(x_1,0,\dots,0)$.
 
 Else assume $J>0$. For $f\in F_0(A,B)^{(k)}$ and $x_1\in A$, define
\[ f_{x_1}\colon A^{k-1} \to B,\ (x_2,\dots,x_k) \mapsto f(x_1,x_2,\dots,x_k). \]
 By assumption the restrictions of $f_{x_1}$ to $(JA)^{k-1}$ can be uniformly generated by $n-1$-ary $JA,\B$-minors.
 That is we have $s\colon \{r\in R^{k-1\times k-1} \st \rk(r) \leq n-1 \} \to S$
 such that
\[ f(x_1,x_2,\dots,x_k) = \sum_{r\in R^{(k-1)\times(k-1)}, \rk(r)\leq n-1} s(r) f(x_1,r(x_2,\dots x_k))\] 
 for all $x_1\in A, x_2,\dots,x_k\in JA.$ Here $(x_1,r(x_2,\dots,x_k))$ denotes the concatenation of 
 the $1$-tuple $x_1$ and the $k-1$-tuple $r(x_2,\dots,x_k)$. For $r\in R^{k-1\times k-1}$ the $k\times k$ block diagonal matrix
 $r' := \begin{pmatrix} 1 & 0 \\ 0 & r \end{pmatrix}$ has inner rank at most $\rk(r)+1$ and satisfies
 $r'(x_1,\dots,x_k) = (x_1,r(x_2,\dots,x_k))$. Then
\[ \hat{f}(x) := \sum_{r\in R^{(k-1)\times(k-1)}, \rk(r)\leq n-1} s(r) f(r' x) \]
 is equal to $f$ on $A\times (JA)^k$. By construction the operator $f\mapsto \hat{f}$ is uniformly represented by $n$-ary
 $\A,\B$-minors on $F_0(A,B)^{(k)}$. 

 Note that $JA$ is the intersection of all maximal submodules $M$ of $\A$.
 Hence applying Lemma~\ref{lem:gM} iteratively for all maximal $M$ to $\hat{f}$ yields that 
 $f\mapsto \hat{f} \mapsto \hat{f}_N$, which is $f_N$, can be uniformly represented by $n$-ary $\A,\B$-minors on
 $F_0(A,B)^{(k)}$. The lemma is proved for $N=A\times (JA)^{k-1}$.

 Finally let $N$ be arbitrary. By Lemma~\ref{lem:free}\eqref{it:transitive} we have some invertible $k\times k$-matrix $T$
 over $\R$ such that $TN=A\times (JA)^{k-1}$. For $f\in F_0(A,B)^{(k)}$, define
\[f'\colon A^k \to B, \ x \mapsto \begin{cases}f(T^{-1}x) & \text{ if } x\in A\times JA^{k-1}, \\
      0 & \text{ else.} \end{cases}\]
 By the proof above, the operator $f\mapsto f'$ can be uniformly represented by $n$-ary $\A,\B$-minors
 on $F_0(A,B)^{(k)}$.
 Since $f_N(x) = f'(Tx)$ for all $x\in A^k$, it follows that $f\mapsto f'\mapsto f_N$ is uniformly represented 
 by $n$-ary $\A,\B$-minors on $F_0(A,B)^{(k)}$ as well.
\end{proof}

 We are now ready to prove the main result of this section. 

\begin{proof}[Proof of Theorem~\ref{thm:Zn}]
 First we assume $\R$ is a finite direct product of finite local rings and prove the result for this case by induction on $A$. 

 Let $k\in\N$. The base case $A=0$ is trivial since all functions in $F(A,B)^{(k)}$ are constant.
 Assume $A\neq 0$.  
 By the induction assumption applied to proper submodules $M$ of $\A$ and Lemma~\ref{lem:uniform}\eqref{it:step}
 it suffices to prove that
\[ F_0(A,B)^{(k)} =  \{f \in F(A,B)^{(k)} \st f(M^k) = 0 \text{ for all } M<\A \} \]
 is uniformly generated by $n$-ary $\A,\B$-minors.

 For this let
\[ V := \{ N\leq\A^k \st (JA)^k\leq N, N/(JA)^k \cong \A/JA \} \]  
 and let $N\in V$.
 Note that if $J>0$, then $J$ has a nontrivial annihilator $I$ in $\R$ and the Jacobson radical of $\R/I$ has nilpotence
 degree $n-1$. Further $JA$ is an $\R/I$-module and $F(JA,B)^{(k)}$ is uniformly generated by $n-1$-ary $\A,\B$-minors
 by the induction assumption.
 So regardless whether $J=0$ or $J>0$, Lemma~\ref{lem:fN} yields that the operator $f\mapsto f_N$ is uniformly generated by
 $n$-ary $\A,\B$-minors on $F_0(A,B)^{(k)}$. The same holds for $f\mapsto \sum_{N\in V} f_N$.
 Since the support of the functions $f_N$ partitions the support of $f\in F_0(A,B)^{(k)}$ by
 Lemma~\ref{lem:free}\eqref{it:cap},\eqref{it:cup},
 we have
\[ f = \sum_{N\in V} f_N. \]
Thus $F_0(A,B)^{(k)}$ is uniformly generated by $n$-ary $\A,\B$-minors.
 The theorem is proved for $\R$ a direct product of local rings.
 
 Next we extend the result to distributive modules $\A$ over arbitrary $\R$. 
 By Lemma~\ref{lem:distributive} we see that $\A$ is also a distributive module over a subring $\R'$ of $\R$,
 which is a finite direct product of finite local rings.
 Let $k\in\N$. Then from the result for distributive $\R'$-modules above we have a function
 $s' \colon \{r \in (\R')^{k\times k} \st \rk(r) \le n\} \rightarrow S$,
 such that for each $f \in F(A,B)^k$ and for all $x \in A^k$
 \[ f(x) = \sum_{r \in (R')^{k \times k}, \rk(r) \le n} s'(r)f(rx). \]
 Clearly $s'$ can be extended to a function $s\colon\{r \in \R^{k \times k} \st \rk(r) \le n\}\to S$, by setting $s$ to be
 $0$ outside of the domain of $s'$. 
 Then for each $f \in F(A,B)^k$ and for all $x \in A^k$
  \[ f(x) = \sum_{r \in R^{k \times k}, \rk(r) \le n} s(r)f(rx). \]
  Thus the result is proved for arbitrary $\R$.
\end{proof}

 Note that the assumptions of Lemma~\ref{lem:fN} are satisfied by (the induction hypothesis in the proof of)
 Theorem~\ref{thm:Zn}. Hence we obtain the following interpolation result that is of independent interest.

\setcounter{theorem}{4} 
\begin{cor} \label{cor:fN}
Let $\R$ be a finite direct product of finite local rings with
 $J := J(\R)$ and $n\in\N$ such that $J^n=0$,
 let $\A$ be a finite distributive $\R$-module, and 
 let $\B$ be an $\S$-module such that $|A|$ is invertible in $\S$.
 Let $k\in\N$ and let $(JA)^k \le N \le \A^k$ such that $N/(JA)^k \cong \A/JA$. 

 Then there exists $s\colon \{r\in R^{k\times k} \st \rk(r) \leq n \} \to S$ such that
 for every $f\colon A^k\to B$ with $f(M^k) = 0$ on all $M<\A$, the function
\[ f_N\colon A^k\to B,\ x\mapsto \begin{cases} f(x) & \text{if } x\in N, \\
     0 & \text{else}, \end{cases} \]
 satisfies 
\begin{equation*} 
  f_N(x) = \sum_{r\in R^{k\times k}, \rk(r)\leq n} s(r) f(rx) \text{ for all } x\in A^k.
\end{equation*}
\end{cor}

 Theorem~\ref{thm:distributive} for modules is an easy consequence of Theorem~\ref{thm:Zn}.

\begin{proof}[Proof of Theorem \ref{thm:distributive} for modules $\A,\B$] 
 Every clonoid from $\A$ to $\B$ is generated by its $n$-ary functions by Theorem~\ref{thm:Zn}.
 If $\B$ is finite, then there are finitely many $n$-ary functions from $A$ to $B$ and so there are finitely many
 clonoids from $\A$ to $\B$.
\end{proof}


\section{Clonoids from distributive abelian Mal'cev algebras}  \label{sec:abelian}

 Extending Theorem~\ref{thm:Zn} we show that every function from a finite abelian Mal'cev algebra $\A$ with distributive
 congruence lattice into some coprime abelian Mal'cev algebra $\B$ is generated by its $n+1$-ary $\A,\B$-minors,
 where $n$ is the nilpotence degree of the Jacobson radical of the ring $\R_\A$ (see the definition in Lemma~\ref{lem:abelian}).
 The increase in the arity compared to the result for modules is due to the fact that we have to compensate for the
 missing constant term function $0$ with a projection, i.e., an additional variable $z$.
 More precisely we have the following
 (recall the definitions for $+_b,*_b,*_z$ from  Lemma~\ref{lem:abelian} and the comments afterwards).

\begin{theorem} \label{thm:abelian}
 Let $\A$ be a finite abelian Mal'cev algebra with distributive congruence lattice,
 let $n\in\N$ be the nilpotence degree of the Jacobson radical of $\R_\A$,
 and let $\B$ be an abelian Mal'cev algebra such that $|A|$ is invertible in $\R_\B$.
  
 For all $k\in\N$ there exists $s\colon \{r\in R_\A^{k\times k} \st \rk(r) \leq n \} \to R_\B$ such that
 for all $f,b\colon A^{k+1}\to B$ and all $x\in A^k,z\in A$
\begin{equation} \label{eq:fabelian}
  f(x,z) = \sum_{r\in R_\A^{k\times k}, \rk(r)\leq n} s(r)*_{b(x,z)} f(r*_zx,z)
\end{equation}
 where the sum is taken pointwise with respect to $+_{b(x,z)}$ in $\B$.
\end{theorem}

\begin{proof}
 Let $k\in\N, a\in A, b\in B$.
 Throughout this proof every sum $\Sigma$ is taken with respect to $+_b$ of $\B_b$ except in two instances
 which we will point out below.
 Since $\A$ and the $\R_\A$-module $\A_a$ are polynomially equivalent by
 Lemma~\ref{lem:abelian}\eqref{it:poly}, they have the same congruences.
 Since the congruence lattice of $\A$ is distributive by assumption, so is the congruence lattice of $\A_a$.
 Now the distributive $\R_\A$-module $\A_a$ and the $\R_\B$-module $\B_b$ satisfy the assumptions of Theorem~\ref{thm:Zn}.
 Thus we have some $s\colon \{r\in R_\A^{k\times k} \st \rk(r) \leq n \} \to R_\B$ 
 such that for all $g\colon A^{k}\to B$ and all $x\in A$
\begin{equation} \label{eq:gAa}
 g(x) = \sum_{r\in R_\A^{k\times k}, \rk(r)\leq n} s(r) *_b g(r*_ax).
\end{equation}
 Note that this formula depends on the constants $a,b$ via the action $*_a$ of $R_\A^{k\times k}$ on $A^k$
 and the underlying addition $u+_av = m(u,a,v)$ of $\A_a$ as well as the operations $*_b,+_b$ of $\B_b$ for the outer sum.
 Moreover, the function $s$ may a priori depend on the choice of $a$ and $b$.

 We claim that $s$ is actually independent of $a,b$ since the modules $\A_a,\B_b$ are isomorphic for all $a\in A, b\in B$,
 respectively.
 To see that, first let $g\colon A^{k}\to B$, $z\in A$ and $p\colon \A_a^k\to\A_z^k$ be the $\R_\A$-module isomorphism
 which exists by Lemma~\ref{lem:abelian}\eqref{it:iso}.
 Setting $g$ to $gp$ in~\eqref{eq:gAa} and using that $p$ is a homomorphism, we obtain for all $x\in A^k$
\[
 gp(x) = \sum_{r\in R_\A^{k\times k}, \rk(r)\leq n} s(r) *_b gp(r*_ax) = \sum_{r\in R_\A^{k\times k}, \rk(r)\leq n} s(r) *_b g(r*_zp(x)).
\]  
 Since $p$ is a bijection on $A^k$, this yields
\begin{equation} \label{eq:gAz}
 g(x) = \sum_{r\in R_\A^{k\times k}, \rk(r)\leq n} s(r) *_b g(r*_zx)
\end{equation}
 for all $x\in A^k$. 
 
 For $f\colon A^{k+1}\to B$ and $z\in A$, we set $g(x_1,\dots,x_k) := f(x_1,\dots,x_k,z)$ in~\eqref{eq:gAz}
 to obtain
\begin{equation} \label{eq:fxz}
 f(x,z) = \sum_{r\in R_\A^{k\times k}, \rk(r)\leq n} s(r)*_{b} f(r*_zx,z)
\end{equation}  
 for all $x\in A^k, z\in A$.
 
 Next let $c\in B$ and let $h\colon \B_c\to\B_b$ be the $\R_\B$-module isomorphism given by
 Lemma~\ref{lem:abelian}\eqref{it:iso}. Setting $f$ to $hf$ in~\eqref{eq:fxz} and using that $h$ is a homomorphism,
 we obtain for all $x\in A^k,z\in A$ that
\begin{align*}
 hf(x,z) & = \sum_{r\in R_\A^{k\times k}, \rk(r)\leq n} s(r)*_{b} hf(r*_zx,z) \\ & = h\left(\sum_{r\in R_\A^{k\times k}, \rk(r)\leq n} s(r)*_{c} f(r*_zx,z)\right).
 \end{align*}
 Note that here the first $\Sigma$ outside of $h$ uses $+_b$ of $\B_b$ and the second $\Sigma$ inside of $h$ uses $+_c$
 in $\B_c$.
 Since $h$ is a bijection on $B$, it follows that
\begin{equation} \label{eq:fxzc}
 f(x,z) = \sum_{r\in R_\A^{k\times k}, \rk(r)\leq n} s(r)*_{c} f(r*_zx,z)
\end{equation}  
for all $x\in A^k,z\in A, c\in B$ where $\Sigma$ again uses $+_c$.
 In particular by choosing $c = b(x,z)$ independently for any point $x\in A^k,z\in A$ in equation~\eqref{eq:fxzc},
 we obtain~\eqref{eq:fabelian}.
\end{proof}

 Theorem~\ref{thm:distributive} is an easy consequence of Theorem~\ref{thm:abelian}.

\begin{proof}[Proof of Theorem \ref{thm:distributive}]
 Let $\A$ be polynomially equivalent to a distributive $\R_\A$-module with $n$ the nilpotence degree of $J(\R_\A)$,
 let $\B$ be polynomially equivalent to an $\R_\B$-module such that $|A|$ is invertible in $\R_\B$.  
 Then every clonoid from $\A$ to $\B$ is generated by its $n+1$-ary functions by Theorem~\ref{thm:abelian}. 

 Hence, if $\B$ is finite, then there are only finitely many clonoids from $\A$ to $\B$. 
\end{proof}

 We point out one important special case of Theorem~\ref{thm:distributive}.

\begin{cor}
 Let $\A = \A_1\times\dots\times\A_\ell$ 
 for pairwise non-isomorphic finite abelian simple Mal'cev algebras $\A_1,\dots,\A_\ell$
 and let $\B$ be a finite abelian Mal'cev algebra such that $|A|$ and $|B|$ are coprime.

 Then every clonoid from $\A$ to $\B$ is generated by binary functions.
\end{cor}

\begin{proof}
 Since the factors $\A_i$ are pairwise non-isomorphic and simple, every congruence of $\A$ is a product congruence.
 That is, the congruence lattice of $\A$ is Boolean, in particular, distributive.
 Since $\A$ is polynomially equivalent to a direct product of simple $\R_\A$-modules, $\R_\A$ is semisimple.
 Thus $\A$ satisfies the assumptions of Theorem~\ref{thm:distributive} with $n=1$ and the result follows.
\end{proof}


\section{Lower bounds on the number of clonoids from $\A$ to $\B$} \label{sec:number}

 In this section we give necessary conditions on the algebras $\A$ and $\B$ such that there are only finitely many
 clonoids between them. 
 We start by showing that if there exists $n\in\N$ such that every clonoid from $\A$ to a nontrivial $\B$ is generated by
 $n$-ary functions, then all subalgebras of $\A$ are generated by $n$ elements.

\begin{lemma} \label{noncyclic}
 Let $n\in\N$, let $\A$ be a finite algebra not all of whose subalgebras are generated by $n$ elements,
 and let $\B$ be a nontrivial algebra.
 Then not every clonoid from $\A$ to $\B$ is generated by $n$-ary functions.
\end{lemma}

\begin{proof}
 By assumption $\A$ has some subalgebra $U$ that is generated by $n+1$ elements but not by $n$.
  
 We construct clonoids $C$ and $D$ from $\A$ to $\B$ with $C^{(k)} = D^{(k)}$ for all $k\leq n$ but $C^{(n+1)} \ne D^{(n+1)}$.
 Let $\rho_U := U^2$ and let
\[ C := \Pol(\rho_U, =_B)\]
 be the clonoid of all functions from $A$ to $B$ that are constant on $U$.
 For $a,b\in A^n$, let $\sigma_{a,b}$ be the subalgebra of $\A^2$ that is generated by $(a_1,b_1),\dots,(a_n,b_n)$. Let
\[ D := \bigcap_{a,b\in U^n} \Pol(\sigma_{a,b},=_B). \]
 Then $f\colon A^k\to B$ is in $D$ iff $f(g_1(a),\dots,g_k(a)) = f(g_1(b),\dots,g_k(b))$ for all
 $a,b\in U^n$ and all $n$-ary term operations $g_1,\dots,g_k$ of $\A$.
 Clearly $C \subseteq D$. 

 To show $D^{(k)} \subseteq C^{(k)}$ for $k\leq n$, let $f\in D^{(k)}$ and $a,b\in U^k$.
 Repeating the last entry of $a,b$, respectively, we obtain $n$-tuples
 $a' := (a_1,\dots,a_k,\dots a_k), b' :=  (b_1,\dots,b_k,\dots b_k)$ over $U$.
 Since $(a_1,b_1),\dots, (a_k,b_k)$ are in $\sigma_{a',b'}$ and $f\in\Pol(\sigma_{a',b'},=_B)$,
 we have $f(a)=f(b)$. Thus $f \in C$ and $C^{(k)} = D^{(k)}$ for all $k\leq n$.

 To show that $C^{(n+1)} \neq D^{(n+1)}$ let $0,1$ be distinct elements in $B$ and define
\[ g \colon A^{n+1} \to B,\ (x_1,\dots,x_{n+1}) \mapsto \begin{cases} 1 & \text{ if } x_1, \ldots, x_{n+1} \text{ generate } U, \\
     0 & \text{ else.} \end{cases} \]
 Since there exist $n+1$ elements that generate $U$ by assumption, $g$ is not constant on $U$, hence not in $C$.
 However $g$ is constant $0$ on all proper subalgebras of $U$ and in particular on all $n$-generated subalgebras.
 Thus $g\in D^{(n+1)} \setminus C^{(n+1)}$. 
\end{proof}

 Even if $\A$ is a cyclic module, not all clonoids from $\A$ to $\B$ may be generated by unary functions by the following.
 
\begin{lemma} \label{JacobsonRadical}
 Let $\R$ be a finite ring with nonzero Jacobson radical, let $\A$ be the left regular module over $\R$, and let
 $\B$ be a nontrivial algebra.
 Then not every clonoid from $\A$ to $\B$ is generated by unary functions.
\end{lemma}

\begin{proof}
 Since $\R$ is a finite ring, its Jacobson radical $J := J(\R)$ is nilpotent.
 It follows that $\R$ has a nonzero ideal $I\leq J$ such that $I^2=0$.
 Note that $I$ is contained in every maximal left ideal of $\R$ since $J$ is the intersection of all maximal left ideals of $\R$.

 We construct clonoids $C,D$ from $\A$ to $\B$ such that $C^{(1)} = D^{(1)}$ but $C^{(2)} \neq D^{(2)}$.
 For any submodule $M$ of $\A$, let $\rho_M := M^2$. Then
 \[ C := \Pol( \equiv_I, =) \cap \bigcap_{M < \A} \Pol(\rho_M,=)  \]
 is the clonoid of functions from $\A$ to $\B$ that are constant on blocks modulo $I$ and constant on every proper
 submodule of $\A$.

 For $a\in I$, let $\sigma_a := R(1,1+a)$ be a submodule of $\A^2$ 
 and let
\[ D := \bigcap_{a\in I} \Pol(\sigma_a,=) \cap \bigcap_{M < \A} \Pol(\rho_M,=). \] 
 Since $\sigma_a$ is contained in $\equiv_I$, we have $\Pol(\equiv_I, =) \subseteq \Pol(\sigma_a, =)$ for all $a\in I$.
 Hence
\begin{equation} \label{eq:CinD}
 C \subseteq D.
\end{equation}
 We claim that
\begin{equation} \label{JacobsonRadicalClaim}
  C^{(1)} = D^{(1)}. 
\end{equation}
 For the proof let $f \in D^{(1)}$ and $x\in A$.
 First consider the case that $Rx\neq R$. Then $x$ is contained in some maximal submodule $M$ of $\A$ and
 $I\leq M$ since $I$ is contained in every maximal submodule. Hence $x+I \subseteq M$ implies $f(x+I)=f(x)$. 

 Next suppose that $Rx = R$. Then $x$ has a multiplicative left inverse in $R$. Since $I$ is finite, $xI = I$.
 Since $f(x) = f(x+xI)$ by assumption, we see that $f(x) = f(x+I)$.

 Hence $f$ is constant on $I$-cosets and $f\in C$. Together with~\eqref{eq:CinD} this implies~\eqref{JacobsonRadicalClaim}.

 Next we claim that
 \begin{equation} \label{eq:C2neD2}
  C^{(2)} \ne D^{(2)}.
\end{equation}
 To see this let $0,1$ be distinct elements in $B$ and define
 \[g \colon A^2 \rightarrow B, \, (x,y) \mapsto \begin{cases} 1 & \text{ if } x=y \equiv_I 1, \\
     0 & \text{ else.} \end{cases} \]
 Clearly $g \notin C$ as for nonzero $a \in I$ we have $(1,1) \equiv_I (1+ a, 1)$ but $g(1,1) = 1 \ne 0 = g(1+a,1)$. 

 For $g \in D$ we first show that $g \in \Pol(\sigma_a,=)$ for all $a\in I$. Let $(x,y)\in A^2$ and $a\in I$.

 If $x\not\equiv_I 1$, then also $x(1+a) = x+xa \not\equiv_I 1$ and $g(x,y) = 0 = g(x(1+a),y(1+a))$.
 The symmetric argument works if $y\neq_I 1$.
 Hence we assume $x,y\in 1+I$ in the following, say $x=1+u$ and $y = 1+v$ for some $u,v\in I$.
 If $u\neq v$, then $g(x,y)=0$ and also
\[ g(x(1+a), y(1+a)) = g(1+u+a+\underbrace{ua}_{=0}, 1+v+a+\underbrace{va}_{=0}) = 0. \]
 Here we used the assumption that $I^2=0$. 
 If $u = v$, then similarly $g(x,y)=1$ and
\[ g(x(1+a), y(1+a)) = g(1+u+a, 1+u+a) = 1. \]
 Hence $g \in \Pol(\sigma_a,=)$.

 Now let $M$ be a proper submodule of $\A$ and show $g\in\Pol(\rho_M,=)$. Note that by the nilpotence of $J$,
 every element in $1+J$ is a unit in $\R$. Hence in particular $1+I$ and $M$ are disjoint. Thus $g(M,M) = 0$.

 Therefore $g\in D^{(2)} \setminus C^{(2)}$ proving~\eqref{eq:C2neD2} and that $D$ is not generated by its unary functions.
\end{proof}

 Next we prove Theorem \ref{CommutativeSummary} which characterizes the finite modules over commutative
 rings for which all clonoids are generated by their unary functions.

 \begin{proof}[Proof of Theorem \ref{CommutativeSummary}.]

  The implication $(2) \Rightarrow (1)$ follows from Theorem~\ref{thm:distributive}. 

 For $(1) \Rightarrow (2)$, assume that every clonoid from the $\R$-module $\A$ to $\B$ is generated by unary functions.
 By Lemma~\ref{noncyclic} $\A$ is a cyclic module. So $\A$ is isomorphic to $R/L$ for some left ideal $L$ of $\R$.
 Since $\R$ is commutative, $L$ is a (two-sided) ideal of $\R$. Now $L (\R/L) = 0$ yields $LA = 0$ and $L=0$ since $\A$ is a
 faithful $\R$-module. Hence $\A$ is isomorphic to the regular $\R$-module. Since the Jacobson radical $J(\R)=0$ by
 Lemma~\ref{JacobsonRadical} and $\R$ is finite, the Wedderburn-Artin Theorem yields that $\R$ is isomorphic to a direct product of
 matrix rings over finite fields. Since $\R$ is commutative, each matrix ring has dimension $1 \times 1$ and $\R$ is a direct product
 of finite fields.
\end{proof}

 Finally we construct an infinite ascending chain of clonoids between any two finite modules whose orders
 have a nontrivial common divisor to prove Theorem~\ref{thm:notcoprime}.
 So in this case we attain the upper bound guaranteed by Theorem~\ref{thm:cube}.

\begin{proof}[Proof of Theorem \ref{thm:notcoprime}.]
 Let $p$ be a common prime divisor of $|A|$ and $|B|$. Then $\A$ has a simple quotient $\A'$ of $p$-power order and
 $\B$ has a simple submodule $\B'$ of $p$-power order. By the Jacobson Density Theorem $\A'$ and $\B'$ are
 modules over full matrix rings over (not necessarily the same) finite fields of characteristic $p$.
 In particular we can expand $\A'$ to a module $\A^+ \cong \Z_p^m$ over the matrix ring $\Z_p^{m\times m}$ and $\B'$ to a module
 $\B^+ \cong \Z_p^n$ over $\Z_p^{n\times n}$ for some natural numbers $m$ and $n$. 

 For ease of notation assume $A^+ = \Z_p^m$ and $B^+ = \Z_p^n$ in the following. Let $D$ be the clonoid of all additive functions
 from $\A^+$ to $\B^+$.
 For $k\in\N$ let
\[ f_k\colon (A^+)^k\to B^+,\ (x_1,\dots,x_k) \mapsto (x_{1,1}x_{2,1} \cdots x_{k,1},\underbrace{0,\dots,0}_{n-1}). \] 
 Here $x_{1,1}x_{2,1} \cdots x_{k,1}$ denotes the usual product in $\Z_p$.
 We claim that
\begin{equation} \label{sequence}
 D \subsetneq D+\langle f_2 \rangle \subsetneq D+\langle f_2, f_3 \rangle \subsetneq \cdots
\end{equation}
 is an infinite strictly ascending sequence of clonoids from $\A^+$ to $\B^+$.

 To see this let $\pi_i\colon\Z_p^n\to \Z_p$ denote the $i$-th projection for $i\in [n]$.
 From the definition of $f_k$ we see that $\pi_1 f_k$ is a polynomial function on $\Z_p$ in variables $x_{1,1},\dots,x_{k,m}$
 of total degree $k$.
 It follows that for every $g \in D+\langle f_2, \ldots, f_k \rangle$ the components 
 $\pi_i g$ are induced by polynomials of total degree at most $k$ over $\Z_p$ for all $i\in [n]$.
 In particular $f_{k} \not\in D+\langle f_1, \ldots, f_{k-1} \rangle$ for $k\geq 2$ and~\eqref{sequence} is proved.

 Now let $E$ be the clonoid of all additive functions from $\A$ to $\B$. For $k\geq 2$, let $C_k$ be the clonoid from $\A$ to
 $\B$ that is obtained by lifting the functions in $\langle f_2,\dots,f_k \rangle_{\A^+,\B^+}$ by Lemmas~\ref{lem:expansions}
 and~\ref{quotients}. Conversely, by restricting to the domain $\A^+$ and codomain $\B^+$, we see that $E+C_k$ induces
 $D+\langle f_2,\dots,f_k \rangle_{\A^+,\B^+}$. Hence~\eqref{sequence} yields that
\[ E \subsetneq E+C_2 \subsetneq E+C_3 \subsetneq \cdots \]
 is an infinite strictly ascending series of clonoids from $\A$ to $\B$.
\end{proof}

\section*{Acknowledgment}

 The authors thank the anonymous referee for their helpful comments on the presentation of this paper.


\end{document}